\documentclass[reqno, 12pt]{amsart}
\usepackage[utf8]{inputenc}
\usepackage{amssymb,amsmath,amsfonts,amsthm,calrsfs,mathtools}
\usepackage[all]{xy}
\usepackage{color}
\usepackage[english]{babel}
\usepackage[T1]{fontenc}
\usepackage{latexsym}
\usepackage{enumitem}
\usepackage{tabularx}
\usepackage[a4paper,top=3cm,bottom=2cm,left=3cm,right=3cm,marginparwidth=1.75cm]{geometry}
\usepackage[colorinlistoftodos]{todonotes}
\usepackage[colorlinks=true, allcolors=red]{hyperref}

\theoremstyle{plain}
\newtheorem{theorem}{Theorem}
\newtheorem{lemma}{Lemma}

\newtheorem{proposition}{Proposition}

\newtheorem{corollary}{Corollary}

\theoremstyle{definition}

\newtheorem{definition}{Definition}

\newtheorem{remark}{Remark}

\providecommand{\keywords}[1]

\title{On the dynamics of a semigroup and its relation with the Riemann Hypothesis}
\author{Carlos F. \'{A}lvarez}

\address{Departamento de Matem\'{a}ticas, Universidad del Atl\'{a}ntico, Cra 30 \# 8-49, Puerto Colombia, Colombia}

\email{cfalvarez@mail.uniatlantico.edu.co \ (Carlos F. \'Alvarez)}
\urladdr{0000-0001-5717-6531}

\author{Juan Manzur}

\email{jcmanzur@mail.uniatlantico.edu.co \ (Juan Manzur) \ Corresponding Author}
\urladdr{0000-0003-3659-6574}

\makeatletter

\date{\today}

\begin{document}

\begin{abstract}

The semigroup of weighted composition operators $(W_n)_{n\in \mathbb{N}}$, defined by $$W_nf(z)=(1+z+\cdots +z^{n-1})f(z^n),$$ acts on the classical Hardy-Hilbert space $H^{2}(\mathbb{D})$, and exhibits intriguing connections with both the Riemann Hypothesis (RH) and the Invariant Subspace Problem (ISP). In this paper, we prove that the adjoint operators $W^{\ast}_{n}$, for $n\geq 2$, are Devaney chaotic, frequently hypercyclic and mixing. In particular, these operators are hypercyclic; we also discuss connections with the RH and invariant subspaces. 
\end{abstract} 

\maketitle
\markright{DYNAMICS OF A SEMIGROUP OF WEIGHTED COMPOSITION OPERATORS $\mathcal{W}$}


\section{Introduction}

Given a separable Banach space $Y$ and a linear operator $T: Y \rightarrow Y$, the \textit{orbit} of a point $f \in Y$ under $T$ is the sequence of iterates $\operatorname{Orb}(f, T) := \{T^n f : n \in \mathbb{N}\}$. A point $f$ is called \textit{periodic} if $T^n f = f$ for some $n \in \mathbb{N}$; the smallest such $n$ is its \textit{period}.
 
Three fundamental notions describe how ``chaotic'' an operator can be. An operator $T$ is \textit{hypercyclic} if some orbit is dense in $Y$, meaning the iterates of a single point eventually come arbitrarily close to every point in the space. It is \textit{Devaney chaotic} if, in addition, periodic points are dense in $Y$. It is \textit{mixing} if any two open regions of $Y$ are eventually connected by iterates of $T$. These notions are central to linear dynamics, a field at the interface of operator theory and dynamical systems; see \cite{BayartMatheron, GrossePeris} for a thorough introduction.
 
A stronger notion, \textit{frequent hypercyclicity}, was introduced by Bayart and Grivaux \cite{bayart2006frequently}. Beyond requiring a dense orbit, it demands that the orbit visits every open set with positive frequency, that is, the returns are not just eventual but statistically regular. This connects linear dynamics to ideas from ergodic theory, the branch of mathematics concerned with long-term statistical behavior of dynamical systems.
 
A natural question is whether chaos and frequent hypercyclicity are equivalent. The answer is no: Bayart and Grivaux \cite{bayart2006frequently} constructed a frequently hypercyclic operator that is not chaotic, and Menet \cite{menet2016chaotic} exhibited a chaotic operator that is not frequently hypercyclic. Although these properties can coincide in special cases \cite{BayartMatheron, DarjiPires21}, they are generally independent. Frequent hypercyclicity does, however, imply several other dynamical properties, including weak mixing and two forms of Li-Yorke chaos \cite{BayartMatheron, Bermudez2011, Bernardes2015}. The relationships among these notions are summarized in the diagram below.
 
\[\xymatrix{
   \textbf{Frequent hypercyclic} \ar@{=>}[dr]  & & 
   \textbf{Devaney chaotic} \ar@{=>}[dl]  \\
\textbf{Mixing} \ar@{=>}[r] & \textbf{Weak mixing} \ar@{=>}[r] 
& \textbf{Hypercyclic} \ar@{=>}[d] \\
 & \textbf{Li-Yorke chaotic}  & \ar@{=>}[l] 
 \textbf{Densely Li-Yorke chaotic} 
}\]
 
We focus on the semigroup of weighted composition operators $\mathcal{W} = (W_n)_{n \in \mathbb{N}}$ on the Hardy space $H^2 = H^2(\mathbb{D})$, defined by
\begin{equation}\label{WaleedSemigroup}
    W_n f(z) = \frac{1 - z^n}{1 - z} f(z^n),
\end{equation}
first introduced in \cite{Noor}, satisfying $W_m W_n = W_{mn}$ and $W_1 = I_{H^2}$. It was shown in \cite{Manzur} that each adjoint $W_n^*$ ($n \geq 2$) is universal in the sense of Rota with spectrum $\sigma(W_n^*) = B(0,\sqrt{n})$. The family $(W_n^*)_{n \geq 2}$ admits special eigenvectors, the \textit{zeta kernels} $\kappa_s \in H^2$ ($\Re s > 1/2$), satisfying $W_n^* \kappa_s = n^{1-s} \kappa_s$ \cite{Calderaro2024, Ghosh2024}. The distribution of these eigenvalues across the regions $|n^{1-s}| \gtrless 1$ and $|n^{1-s}| = 1$ is the key input for the Godefroy--Shapiro criterion. Remarkably, this single family satisfies mixing, Devaney chaos, and frequent hypercyclicity simultaneously three generally independent properties and no prior dynamical results on $\mathcal{W}$ had been reported.

The semigroup $\mathcal{W}$ is also connected to the Riemann Hypothesis (RH), which asserts that all nontrivial zeros of the Riemann zeta function lie on the critical line $\Re(s) = \tfrac{1}{2}$. Setting $h_k(z) = \frac{1}{1-z}\log\!\left(\frac{1+z+\cdots+z^{k-1}}{k}\right)$ and $\mathcal{N} = \mathrm{span}\{h_k : k \geq 2\}$, the following result due to Noor \cite{Noor} makes this precise.

\begin{theorem}[{\cite{Noor}}]\label{Noor1}
The following statements are equivalent:
\begin{enumerate}
    \item The RH is true.
    \item The constant function $1$ belongs to the closure of $\mathcal{N}$ in $H^2$.
    \item The subspace $\mathcal{N}$ is dense in $H^2$.
    \item The closure of $\mathcal{N}$ contains a cyclic vector for $\mathcal{W}$.
\end{enumerate}
\end{theorem}

Equivalently, the RH holds if and only if $\mathcal{N}^{\perp} = \{0\}$, motivating the study of subspaces $X \subseteq H^2$ for which
\begin{align}\label{c1}
\mathcal{N}^{\perp} \cap X = \{0\}.
\end{align}

The main goal of this paper is to prove that $W_n^*$ is mixing, Devaney chaotic, and frequently hypercyclic for each $n \geq 2$, using the zeta kernels $\{\kappa_s\}$ together with the Godefroy--Shapiro criterion and the eigenvector criterion of Bayart and Matheron, and to explore the implications for the RH and invariant subspaces. Section~\ref{Preli} reviews background results. Section~\ref{Mixingchaos} establishes mixing and Devaney chaos. Section~\ref{Frequenthyp} proves frequent hypercyclicity. Section~\ref{RHIS} discusses the RH and the existence of a dense subspace of hypercyclic vectors for each $W_n^*$ satisfying~(\ref{c1}).

\section{Definition and background results}\label{Preli}
In this section, we establish the notation and recall preliminary concepts concerning the Hardy-Hilbert space $H^2$ and the theory of linear dynamics. We denote by $\sigma(T)$ the \emph{spectrum} of a linear operator $T$, and by $\sigma_p(T)$ its \emph{point spectrum}. Given $r > 0$, we write $\overline{B}(a, r)$ the closed ball of radius \( r \) centered at \( a \in \mathbb{C} \).

\subsection{The weighted composition semigroup  \texorpdfstring{$\mathcal{W}$}{lg}}

Let $\mathbb{D} = \{ z \in \mathbb{C} : |z| < 1 \}$ and  $\mathbb{T} = \{ z \in \mathbb{C} : |z| = 1 \}$ denote the open  unit disk and unit circle in the complex plane, respectively. The Hardy space \( H^2 := H^2(\mathbb{D}) \) is the space of all analytic functions in \( \mathbb{D} \) that admit a power series representation of the form  $f(z) =\sum_{n=0}^{\infty} \hat{f}(n) z^n$, for which the norm $\|f\|^2_{2} := \sum_{n=0}^{\infty} |\hat{f}(n)|^2 < \infty.$ This space is a \textit{Hilbert space} with inner product defined by:
\[
\langle f, g \rangle = \sum_{n=0}^{\infty} \hat{f}(n)\overline{\hat{g}(n)}.
\]
where $f,g\in H^2$. Related to this, the space $H^\infty$ consists of bounded holomorphic functions on $\mathbb{D}$:
\[ H^\infty := \left\{ f \colon \mathbb{D} \to \mathbb{C}:\ f \text{ holomorphic and } \|f\|_{\infty} :=\sup_{z \in \mathbb{D}} |f(z)| < \infty \right\}.\]

\begin{definition}
Given an analytic self-map \( \varphi: \mathbb{D} \to \mathbb{D} \) and an analytic function \( w: \mathbb{D} \to \mathbb{C} \), the \emph{weighted composition operator} \( C_{w,\varphi} \) is defined on \( H^2 \) by
\[C_{w,\varphi}f(z) = w(z) \cdot f(\varphi(z)), \quad f \in H^2, \, z \in \mathbb{D}.\]
\end{definition}

\begin{remark}
If the weight function \( w \) belongs to \( H^\infty \), then \( C_{w,\varphi} \) is a bounded operator on \( H^2 \), see \cite{Matache}.
\end{remark}

We are going to focus this work on the study of the semigroup of weighted composition operators 
\(\mathcal{W} = (W_n)_{n \in \mathbb{N}}\) on \(H^2\) defined in \eqref{WaleedSemigroup}.

We first establish an adjoint formula for the elements of the semigroup \(\mathcal{W}\) which will be used throughout the rest of this work.

\begin{lemma}[See {\cite[Lemma~3]{Manzur}}]\label{adjointformula}
For any \(f \in H^2\) and \(n \in \mathbb{N}\), the adjoint of \(W_n\) is given by
\[W_n^* f(z) = \sum_{k=0}^{\infty} B_n(k) z^k, \]
where \((\hat{f}(k))_{k \in \mathbb{N}}\) are the Maclaurin coefficients of \(f\) and
\[ B_n(k) = \hat{f}(nk) + \hat{f}(nk + 1) + \cdots + \hat{f}(nk + n - 1)\]
is the sum of the \(k\)-th block of \(n\) consecutive coefficients.
\end{lemma}

As a consequence of Lemma \ref{adjointformula}, it follows that this semigroup satisfy the property that
\begin{equation*}\label{Wident}
W_n^* W_n = n I_{H^2}
\end{equation*}
for all \( n \in \mathbb{N} \), and therefore that \( (W_n / \sqrt{n})_{n \in \mathbb{N}} \) is a semigroup of \emph{isometries}. Moreover, \( (W_n / \sqrt{n})_{n \in \mathbb{N}} \) defines a \textit{shift semigroup}. Recall that an operator $S\in \mathcal{L}(Y)$ is a shift operator if \( \|S^{*n} f\| \to 0 \) for all \( f \in Y \). As a consequence of \cite[Proposition~4]{Manzur}, every \(W_n^*\) is universal in the sense of Rota and we have the following consequence. 

\begin{theorem}[{See \cite[Corollary 5]{Manzur}}]\label{EspMan} $\sigma(W_n^*) = \sigma(W_n) = \overline{B}(0, \sqrt{n}) \quad \text{for all } n \geq 2.$ Also, each \( \sigma_p(W_n) = \emptyset \), and hence \( W_n \) has no finite dimensional invariant subspaces.
\end{theorem}



\subsection{Notions of linear dynamics}

Let $\mathcal{L}(Y)$ be the space of all linear operators $T: Y \rightarrow Y$. The notions of hypercyclicity, Devaney chaos, and mixing have been previously defined. We begin this section with a precise definition of frequent hypercyclicity.

\begin{definition}
Let $A$ be a subset of natural numbers. The \textit{lower density} of \( A \) is defined by
\[\underline{\mathrm{dens}}(A) := \liminf_{N \to \infty} \frac{\mathrm{card}(A \cap [1, N])}{N}.\]
\end{definition}
Let us define 
\[\mathrm{N}(x, V) := \{ n \in \mathbb{N} \, ; \, T^n(x) \in V \}.\]

\begin{definition}
Let \( T \in \mathcal{L}(Y) \). The operator \( T \) is said to be \textit{frequently hypercyclic} if there exists some vector \( y \in Y \) such that \( N(y, V) \) has positive lower density for every non-empty open set \( V \subset Y \). Such a vector \( x \) is said to be frequently hypercyclic for \( T \), and the set of all frequently hypercyclic vectors for \( T \) is denoted by \( \mathrm{FHC}(T) \).
\end{definition}

As consequence of \cite[Theorem~9.8]{GrossePeris} every frequently hypercyclic operator is weakly mixing. 

\begin{definition}
An operator $T\in\mathcal{L}(Y)$ is \textit{Li-Yorke chaotic} if it has an uncountable \textit{scrambled set} $S$, i.e., for all $x, y \in S$ with $x\neq y$, we have that $$\displaystyle\liminf_{n\to \infty}\|T^{n}x-T^{n}y\|=0\ and 
 \ \displaystyle\limsup_{n\to \infty}\|T^{n}x-T^{n}y\|=\infty.$$
If the set \( S \) can be chosen to be dense in \( Y \), then \( T \) is \textit{densely Li-Yorke chaotic}.  \end{definition}

It is well known that every hypercyclic operator is Li-Yorke chaotic (see \cite[Remark~22]{Bernardes2015}). In particular, every frequently hypercyclic operator is Li-Yorke chaotic. 



 \begin{proposition}[{See \cite[Proposition 11]{Bernardes2015}}]
If \( T \in \mathcal{L}(Y) \) and \( T^* \) has an eigenvalue \( \lambda \) with \( |\lambda| \geq 1 \), then \( T \) is not densely Li-Yorke chaotic.
\end{proposition}

As a consequence of the Proposition above we can see that $W_k$, for each $k\geq1$, \textit{is not densely Li-Yorke chaotic} since $W_k^*1=1$, $\forall k\geq1$.  The natural question that arises is the following: Is each operator $W_k$, $k\geq2$, Li–Yorke chaotic? Note first that the three necessary conditions stated in \cite[Corollary 6]{Bermudez2011} are satisfied:
\begin{enumerate}
    \item By Theorem \ref{EspMan}, \( \sigma(W_k) \cap \mathbb{T} \neq \emptyset \).
    \item $W_k$ is not normal since 
    \begin{align*}
        W_kW_k^*1=1+z+\dots+z^{k-1}\neq k=W_k^*W_k1,\ \ \forall k\geq2.
    \end{align*}
    \item  $W_k$ is not a compact operator since it has no eigenvalues and $\sigma(W_k)$ is non-empty (see \cite[Theorem 8.4-4]{Kreyszig}).
\end{enumerate}

These observations suggest that each operator $W_k$, $k\geq2$, is Li–Yorke chaotic. However, in the next section we shall demonstrate that this assertion is false.

\section{A semigroup of mixing and Devaney chaotic operators}\label{Mixingchaos}

In this section we shall prove that the elements of the semigroup $(W_k^*)_{k\geq2}$ are mixing and Devaney chaotic operators. In order to do that, we are going to use the following result:

\begin{theorem}[Godefroy--Shapiro Criterion {\cite[Theorem 3.1]{GrossePeris}}] \label{GSC}
Let \( T \in \mathcal{L}(Y) \). Suppose that
\[\bigcup_{|\lambda| < 1} \ker(T - \lambda I) \quad \text{and} \quad \bigcup_{|\lambda| > 1} \ker(T - \lambda I) \]
both span a dense subspace of  \( Y\). Then \( T \) is mixing, and in particular hypercyclic.\vspace{2mm}

If, moreover, \( Y \) is a complex space and also the subspace
\[Z_0 := \mathrm{span} \{ x \in Y \; ; \; Tx = e^{\alpha \pi i} x \text{ for some } \alpha \in \mathbb{Q} \}\]
is dense in \( Y \), then \( T \) is Devaney chaotic.
\end{theorem}

\begin{remark}\label{PerChar}
Let us denote \(\mathrm{Per}(T) \) the set of periodic points of $T$. In  {\cite[Proposition 2.33]{GrossePeris}} it is shown that \[ \mathrm{Per}(T)= \mathrm{span} \{ x \in X \; ; \; Tx = e^{\alpha \pi i} x \text{ for some } \alpha \in \mathbb{Q} \}.\]
\end{remark}

\subsection{A mixing semigroup}   

We will explore now the eigenvalues of the semigroup \( (W_n^{\ast})_{n \in \mathbb{N}} \).  Let \( H^2(\mathbb{C}_{1/2}) \) be the \textit{Hardy space of the open right half-plane} \( \mathbb{C}_{1/2} = \{ \Re s > 1/2 \} \). This is a Hilbert space of analytic functions on \( \mathbb{C}_{1/2} \) for which the norm is given by
\[ \|f\|_2^2 = \sup_{0 < x < \infty} \frac{1}{\pi} \int_{-\infty}^{\infty} |f(x + iy)|^2 \, dy. \]

In \cite{Ghosh2024} it was introduced a linear functional \( \Lambda^{(s)}:= E_s \circ \Lambda : H^2\to \mathbb{C} \), where $E_s: H^2(\mathbb{C}_{1/2}) \to \mathbb{C}$ is the evaluation functional at \( s \in \mathbb{C} \setminus \{0\} \) and $\Lambda: H^2 \to H^2(\mathbb{C}_{1/2})$ is the result of a composition of certain isometries. \( \Lambda^{(s)}\)  can be defined by assigning
\[
\Lambda^{(s)}(1) = -\frac{1}{s}, \quad \Lambda^{(s)}(z^k) = -\frac{1}{s} \left( (k+1)^{1-s} - k^{1-s} \right) \quad (k \in \mathbb{N}).
\]

In particular \( \Lambda^{(s)} \) is bounded on \( H^2 \)  if \( \Re s > 1/2 \). So there exist functions \( \kappa_s \in H^2 \) such that \( \Lambda^{(s)}(f) = \langle f, \kappa_s \rangle \). The function \( \kappa_s \) is called the zeta kernel at \( s \) and

\begin{equation*}
\kappa_s(z) = \sum_{k=0}^{\infty} \varphi_k(\bar{s}) z^k \quad \text{where} \quad \varphi_k(s) := \Lambda^{(s)}(z^k).
\end{equation*}

In  \cite{Calderaro2024} it was shown that the set $\mathcal{K}=\{ \kappa_s: \Re s > 1/2 \}$ is a family of common eigenvectors for $(W_k^*)_{k\in \mathbb{N}}$ with  \[W_n^* \kappa_s = n^{1 - \overline{s}} \kappa_s, \quad \forall\ n\in \mathbb{N}.\]
In particular, $|n^{1 - \overline{s}}|=n^{1 -\Re s} >1$ for $1/2<\Re s < 1$ and  $|n^{1 - \overline{s}}|<1$ for $\Re s > 1$.

\begin{lemma}\label{ktotal}
$\{ \kappa_s:  1/2<\Re s < 1 \}$ span a dense subspace in $H^2$.
\end{lemma}
\begin{proof}
 Let $f\in\{ \kappa_s:  1/2<\Re s < 1 \}^{\perp}$. Then
\begin{align*}
    \Lambda(f)(s)=\Lambda^{(s)}(f)=\langle f,\kappa_s\rangle=0,\ \ \forall \  1/2<\Re s < 1.
\end{align*}
Since \( \Lambda(f) \) is an analytic function that vanishes on a non-empty open set, it must be identically zero; that is, \( \Lambda(f) = 0 \). Moreover, since \( \Lambda \) is a composition of isometries, it is in particular injective (see \cite{Ghosh2024}). Hence, \( f = 0 \), and the assertion follows.
\end{proof}

\begin{lemma}\label{ktotal1}
$\{ \kappa_s:  \Re s > 1 \}$ span a dense subspace in $H^2$.
\end{lemma}
\begin{proof}
    Follows as Lemma~\ref{ktotal}. 
\end{proof}

\begin{proposition}\label{dense1}
   $\displaystyle\bigcup_{|\lambda| > 1} \ker(W_n^* - \lambda I)$ and  $\displaystyle\bigcup_{|\lambda| < 1} \ker(W_n^* - \lambda I)$ both span a dense subspace of $H^2$ for all $n\geq2$.  
\end{proposition}
\begin{proof}
Since \( \{ \kappa_s : 1/2 < \Re s < 1 \} \subset \displaystyle\bigcup_{|\lambda| > 1} \ker(W_n^* - \lambda I) \) and \( \{ \kappa_s : \Re s > 1 \} \subset \displaystyle\bigcup_{|\lambda| < 1} \ker(W_n^* - \lambda I) \), the result follows from Lemma~\ref{ktotal} and Lemma~\ref{ktotal1}.  
\end{proof}

\begin{proposition}\label{mixingthm}
$W_n^*$ is mixing and, in particular, hypercyclic operator for all $n\geq2$.
\end{proposition}
\begin{proof}
It is a consequence of Proposition \ref{dense1} and the Godefroy-Shapiro Criterion.
\end{proof}

\begin{corollary}
 $W_n^*$ is a densely Li-Yorke chaotic operator for all $n\geq2$   
\end{corollary}
\begin{proof}
Hypercyclicity implies densely Li-Yorke chaos (see \cite{Bernardes2015}).
\end{proof}

In the previous section, we raised the question of whether \( W_k \)  is Li–Yorke chaotic. We now answer this question.

\begin{proposition}
The family \( (W_n)_{n \in \mathbb{N}} \) forms a semigroup of operators that are not Li-Yorke chaotic.
\end{proposition}
\begin{proof}
Recall that $T \in \mathcal{L}(Y)$ is \textit{uniformly positively expansive} if there exists $n \in \mathbb{N}$ such that $\|T^n z\| \geq 2$ for all $z \in S_Y$. Since $W_k/\sqrt{k}$ is an isometry, for any $f \in S_{H^2}$ we have
\[\|W_k^n f\| = \|W_{k^n} f\| = \sqrt{k}^n,\] which shows that $W_k$ is uniformly positively expansive for all $k, n \geq 2$. By \cite[Theorem C]{Bernardes2018}, no such operator can be Li-Yorke chaotic. The case $k = 1$ is trivial since $W_1 = I_{H^2}$.
\end{proof}

\subsection{A Devaney chaotic semigroup} To conclude this section, we will show that the operators \( W_n^* \) are not only hypercyclic but also Devaney chaotic for all \( n \geq 2 \). \vspace{2mm}

First, observe that
\[
\{ \kappa_s : \Re(s) = 1 \} \subset \bigcup_{|\lambda| = 1} \ker(W_n^* - \lambda I),
\]
since for \( s = 1 + it \), with \( t \in \mathbb{R} \), we have \( W_n^* \kappa_s = n^{1 - \overline{s}} \kappa_s = n^{it} \kappa_s \), and \( |n^{it}| = 1 \). Now, noting that \( n^{it} = e^{it \log n} \), we define
\[
t_{n,\alpha} := \frac{\pi \alpha}{\log n}, \quad \alpha \in \mathbb{Q},
\]
so that, by Remark~\ref{PerChar}, we obtain
\[
\kappa_{s_{n,\alpha}} \in \mathrm{Per}(W_n^*) \quad \text{for all } n \geq 2,
\]
where \( s_{n,\alpha} = 1 + i t_{n,\alpha} \). Since \( \mathbb{Q} \) is dense in \( \mathbb{R} \), it follows that \( \{ t_{n,\alpha} \}_{\alpha \in \mathbb{Q}} \) is dense in \( \mathbb{R} \), and therefore \( \{ s_{n,\alpha} \}_{\alpha \in \mathbb{Q}} \) is dense in the vertical line \( \{ \Re(s) = 1 \} \), for each $n\geq2$.

\begin{lemma}\label{k1total}
$\{ \kappa_{s_{n,\alpha}}\}_{\alpha\in\mathbb{Q}}$ span a dense subspace in $H^2$, for all $n\geq2$.
\end{lemma}
\begin{proof}
Let \( f \in \{ \kappa_{s_{n,\alpha}} \}_{\alpha \in \mathbb{Q}}^\perp \). Then
\[
\Lambda(f)(s_{n,\alpha}) = \Lambda^{(s_{n,\alpha})}(f) = \langle f, \kappa_{s_{n,\alpha}} \rangle = 0.
\]
Since the set \( \{ s_{n,\alpha} \}_{\alpha \in \mathbb{Q}} \) is dense in the vertical line \( \{ \Re(s) = 1 \} \), the zero set of \( \Lambda(f) \) has an accumulation point (in fact, every point on the line is an accumulation point). As \( \Lambda(f) \) is analytic, it must vanish identically; that is, \( \Lambda(f) = 0 \). Therefore, \( f = 0 \), and the result follows.
\end{proof}
\begin{theorem}
    $W_n^{\ast}$ is a Devaney chaotic operator for all $n\geq 2.$
\end{theorem}
\begin{proof}
Since \( \{ \kappa_{s_{n,\alpha}} \}_{\alpha \in \mathbb{Q}} \subset \mathrm{Per}(W_n^*) \), it follows that the set of periodic points \( \mathrm{Per}(W_n^*) \) is dense in \( H^2 \). Moreover, Theorem~\ref{mixingthm} establishes that \( W_n^* \) is hypercyclic; therefore, the Godefroy–Shapiro Criterion ensures that \( W_n^* \) is Devaney chaotic.
\end{proof}

\section{A semigroup of frequently hypercyclic operators}\label{Frequenthyp}

This section is devoted to proving that the semigroup \( (W_n^{\ast})_{n \in \mathbb{N}} \) consists of frequently hypercyclic operators. We begin by presenting some preliminary results.

\begin{definition}
We say that a Banach space \( Y \) is of \textit{type \( p \)} for some \( 1 \leq p \leq 2 \) if there exists a constant \( C \geq 1 \) such that
\[
\left\| \sum_{i=1}^n \varepsilon_i x_i \right\|_{L^2(Y)} \leq C \left( \sum_{i=1}^n \|x_i\|^p \right)^{1/p}
\]
for every \( x_1, \ldots, x_n \in Y \), where \( (\varepsilon_n)_{n \in \mathbb{N}} \) is a Bernoulli sequence.
\end{definition}

In \cite{LiQueffelec2017} it is shown that every Hilbert space is type 2.

\begin{definition}
A complex measure \( \sigma \) on \( \mathbb{T} \) is said to be \textit{continuous} if it has no discrete part, i.e. \( \sigma(\{\lambda\}) = 0 \) for every \( \lambda \in \mathbb{T} \). 
\end{definition}

\begin{definition}
Let \( T \in \mathcal{L}(Y) \). Given a probability measure \( \sigma \) on \( \mathbb{T} \), we say that \( T \in \mathcal{L}(Y) \) has a \textit{\( \sigma \)-spanning set of \( \mathbb{T} \)-eigenvectors} if for every \( \sigma \)-measurable subset \( A \subset \mathbb{T} \) with \( \sigma(A) = 1 \), the eigenspaces \( \ker(T - \lambda I) \), \( \lambda \in A \), span a dense subspace of \( Y \). If \( T \) has a \( \sigma \)-spanning set of \( \mathbb{T} \)-eigenvectors for some continuous measure \( \sigma \), then we say that \( T \) has a \textit{perfectly spanning set of \( \mathbb{T} \)-eigenvectors}.
\end{definition}

\begin{proposition}[{See \cite[Corollary 6.24]{BayartMatheron}}]\label{FreqEig}
Let \( Y \) be a complex separable infinite-dimensional Banach space, and let \( T \in \mathcal{L}(Y) \). If \( Y \) is of type 2 and \( T \) admits a perfectly spanning set of \( \mathbb{T} \)-eigenvectors, then \( T \) is frequently hypercyclic.
\end{proposition}

Let us denote $m$ the normalized Lebesgue measure on $\mathbb{T}$, which is a probability measure on this space.

\begin{lemma}
Let \( A \subset \mathbb{T} \) be a \( m \)-measurable subset such that \( m(A) = 1 \). Then \( A \) has no isolated points.
\end{lemma}
\begin{proof}
Suppose that \( p\) is an isolated point in \( A \). Then there exists an open set \( V_p \subset \mathbb{T} \) such that $V_p \cap A = \{p\}$. We also have that
\[
m(\mathbb{T}) = m(\mathbb{T} \setminus A) + m(A).
\]
As \( m(A) = 1 \), it follows that \( m(\mathbb{T} \setminus A) = 0 \), which is a contradiction because \( \mathbb{T} \setminus A \) contains the non-empty open set \( V_p \setminus \{p\} \), which must have positive measure.
\end{proof}

Under the conditions of the previous lemma, every point in \( A \) is an accumulation point. In particular, $A$ cannot be a  finite set. \vspace{2mm}

Let us define the map, for each $n\geq2,$
\[
\varphi_n : \left(0, \frac{2\pi}{\log n}\right) \to \mathbb{T} \setminus \{1\}, \quad \varphi_n(t) = e^{it \log n}.
\]
Each \( \varphi_n \) is a homeomorphism, as it arises from the classical homeomorphism
\[
\varphi : (0, 2\pi) \to \mathbb{T} \setminus \{1\}, \quad \varphi(t) = e^{it}.
\]
\begin{lemma}
    Let \( A \subset \mathbb{T} \) be a \( m \)-measurable subset such that \( m(A) = 1 \). Then $\{ \kappa_{1+it}\}_{t\in \varphi_n^{-1}(A)}$ span a dense subspace in $H^2$, for all $n\geq2$.  
\end{lemma}
\begin{proof}
 Let \( f \in \{ \kappa_{1+it}\}_{t\in \varphi_n^{-1}(A)}^\perp \). Then
\[
\Lambda(f)(1+it) = \Lambda^{(1+it)}(f) = \langle f, \kappa_{1+it} \rangle = 0,\ \ \forall t\in \varphi_n^{-1}(A).
\]
Since \( A \) cannot be a finite set, let us take \( a \in A \setminus \{1\} \). As \( a \) is an accumulation point of \( A \), it follows that \( t_0 = \varphi_n^{-1}(a)\) is an accumulation point of \( \varphi_n^{-1}(A) \) by the continuity. Therefore, the zero set of \( \Lambda(f) \) has an accumulation point. Since \( \Lambda(f) \) is analytic, it must vanish identically; that is, \( \Lambda(f) = 0 \). Hence, \( f = 0 \), and the result follows.
\end{proof}

\begin{proposition}\label{KerMeasure}
    Let \( A \subset \mathbb{T} \) be a \( m \)-measurable subset such that \( m(A) = 1 \). Then  \(\displaystyle\bigcup_{\lambda\in A} \ker(W_n^* - \lambda I) \) span a dense subspace of $H^2$, for all $n\geq2$.
\end{proposition}
\begin{proof}
    Note that if \( t \in \varphi_n^{-1}(A) \), then
    \[
    W_n^* \kappa_{1+it} = n^{it} \kappa_{1+it} = e^{it\log n} \kappa_{1+it}.
    \]
    Since \( \varphi_n(t) = e^{it\log n} \in A \), it follows that
    \[
    \left\{ \kappa_{1+it} \right\}_{t \in \varphi_n^{-1}(A)} \subset \bigcup_{\lambda \in A} \ker(W_n^* - \lambda I),
    \]
    and the assertion follows from the previous Lemma.
\end{proof}

\begin{theorem}
$W_n^{\ast}$ is a frequently hypercyclic operator for all $n\geq 2.$
\end{theorem}
\begin{proof}
Since \( m \) is a continuous measure, Proposition \ref{KerMeasure} implies that \( W_n^{\ast} \) has a perfectly spanning set of \( \mathbb{T}\)-eigenvectors. Then, by Proposition \ref{FreqEig}, the assertion follows.
\end{proof}

\section{Riemann hypothesis and invariant subspaces related 
to the semigroup \texorpdfstring{$(W_k^*)_{k\geq2}$}{lg}}
\label{RHIS}

In \cite{Calderaro2024} it was shown that $W_n$ does not have 
finite-dimensional invariant subspaces for all $n\geq 2$. We 
provide a new proof of this result.

\begin{theorem}
Let $S$ be a non-trivial closed $W_n$-invariant subspace for 
some $n\geq 2$. Then $S$ is infinite-dimensional.
\end{theorem}
\begin{proof}
Since $W_nS\subset S$ if and only if $W_n^*S^\perp\subset 
S^\perp$, and $W_n^*$ is hypercyclic, $S=S^{\perp\perp}$ is 
infinite-dimensional; see \cite[p.~15]{BayartMatheron}.
\end{proof}

Let $\mathcal{N}=\mathrm{span}(h_k)_{k\geq 2}$. By 
Theorem~\ref{Noor1}, the RH is equivalent to 
$\overline{\mathcal{N}}=H^2$. We now give an equivalent 
formulation of the RH in terms of hypercyclic vectors for 
$W_k^*$.

\begin{theorem}\label{thm:RHhyp}
Let $k\geq 2$. The RH holds if and only if there exist a 
hypercyclic vector $f$ for $W_k^*$ and $N\in\mathbb{N}$ such 
that $W_k^{*n}f\in\overline{\mathcal{N}}$ for all $n\geq N$.
\end{theorem}
\begin{proof}
If the RH holds, then $\overline{\mathcal{N}}=H^2$ and the 
condition is trivially satisfied. Conversely, suppose such 
$f$ and $N$ exist. Since $(W_k^{*n}f)_{n\geq 0}$ is dense 
in $H^2$, so is the tail $(W_k^{*n}f)_{n\geq N}$. As this 
tail lies in $\overline{\mathcal{N}}$, we get 
$\overline{\mathcal{N}}=H^2$, and the RH follows.
\end{proof}

There has been interest in identifying subspaces $X\subset H^2$ 
satisfying $\mathcal{N}^\perp\cap X=\{0\}$; see 
\cite{Calderaro2024, Noor}. We show this holds unconditionally 
for $X=HC(W_n^*)$, which is a dense $G_\delta$ subset of $H^2$ 
by Birkhoff's Transitivity Theorem \cite[Theorem~2.19]{GrossePeris}.

The following general result is the key tool.

\begin{proposition}\label{thm:HCperp}
Let $T\in\mathcal{L}(Y)$ be hypercyclic, $V$ a set 
$($subspace$)$ of hypercyclic vectors for $T$ (except for zero), 
and $M\subset Y$ a closed, proper, $T$-invariant subspace. 
Then $M\cap V=\varnothing$ $($resp.\ $M\cap V=\{0\})$.
\end{proposition}
\begin{proof}
If $f\in M\cap V$ is non-zero, then $\mathrm{Orb}(f,T)\subset 
M$, so $M=Y$, contradicting properness.
\end{proof}

Since $\mathcal{N}^\perp$ is a closed $W_n^*$-invariant 
subspace \cite{Noor}, Proposition~\ref{thm:HCperp} gives:

\begin{corollary}\label{cor:Nperp}
For all $n\geq 2$, $\mathcal{N}^\perp\cap HC(W_n^*)=\varnothing$.
\end{corollary}

Define $\mathcal{M}:=\mathrm{span}\{h_k-h_\ell: k,\ell\geq 2\}$. 
This subspace is proper, $W_n$-invariant, satisfies 
$\mathcal{M}\subset\mathcal{N}$ and $\dim\mathcal{M}^\perp\geq 1$, 
and it was shown in \cite{Manzur} that 
$\mathrm{RH}\Longleftrightarrow\dim\mathcal{M}^\perp=1$. 
The same argument yields:

\begin{corollary}
For all $n\geq 2$, $\mathcal{M}^\perp\cap HC(W_n^*)=\varnothing$.
\end{corollary}

Finally, by the Herrero--Bourdon Theorem 
\cite[Theorem~2.55]{GrossePeris}, for any hypercyclic vector 
$f$ for $T$ the subspace
\[
Z_f(T):=\mathrm{span}\,\mathrm{Orb}(f,T)=
\{p(T)f: p\text{ polynomial}\}
\]
consists of hypercyclic vectors except for zero. Applying 
Proposition~\ref{thm:HCperp} as in Corollary~\ref{cor:Nperp}:

\begin{corollary}
Let $f$ be a hypercyclic vector for $W_n^*$, $n\geq 2$. Then
\[
\mathcal{N}^\perp\cap Z_f(W_n^*)=\{0\}
\quad\text{and}\quad
\mathcal{M}^\perp\cap Z_f(W_n^*)=\{0\}.
\]
\end{corollary}



\bibliography{Bibliography}
\bibliographystyle{acm}

\end{document}